\documentclass[oneside,english]{amsart}
\usepackage[T1]{fontenc}
\usepackage[latin9]{inputenc}
\usepackage{amsthm}
\usepackage{amstext}
\usepackage{amssymb}
\usepackage{esint}

\makeatletter
\numberwithin{equation}{section}
\numberwithin{figure}{section}
  \theoremstyle{plain}
  \newtheorem{thm}{\protect\theoremname}[section]
  \newtheorem{prop}[thm]{\protect\propositionname}
  \theoremstyle{plain}
  \theoremstyle{definition}
	\theoremstyle{definition}
  \newtheorem{defn}[thm]{\protect\definitionname}
  \theoremstyle{plain}
  \newtheorem{lem}[thm]{\protect\lemmaname}
  \newtheorem{cor}[thm]{\protect\corollaryname}
  \theoremstyle{plain}
	\newtheorem{exa}[thm]{\protect\examplename}
  \theoremstyle{plain}
	\theoremstyle{remark}
  \newtheorem{rem}[thm]{\protect\remarkname}
\makeatother

\usepackage{babel}
  \providecommand{\definitionname}{Definition}
  \providecommand{\lemmaname}{Lemma}
  \providecommand{\theoremname}{Theorem}
  \providecommand{\corollaryname}{Corollary}
	\providecommand{\propositionname}{Proposition}
	\providecommand{\examplename}{Example}
	\providecommand{\remarkname}{Remark}
	
	\DeclareMathOperator{\loc}{loc}
	
\begin{document}

\title[On the First Eigenvalues of Free Vibrating Membranes]{On the First Eigenvalues of Free Vibrating Membranes in Conformal Regular Domains}

\author{V.~Gol'dshtein and A.~Ukhlov}
\begin{abstract}
In 1961 G.~Polya published a paper about the eigenvalues of vibrating membranes. The "free vibrating membrane" corresponds to the Neumann-Laplace operator in bounded plane domains. In this paper we obtain estimates for the first non-trivial eigenvalue of this operator in a large class of domains that we call conformal regular domains. This case includes convex domains, John domains etc... On the base of our estimates we conjecture that the eigenvalues of the Neumann-Laplace operator depend on the hyperbolic metrics of plane domains. We propose a new method for the estimates which is based on weighted Poincar\'e-Sobolev inequalities, obtained by the authors recently.

\end{abstract}
\maketitle
\footnotetext{\textbf{Key words and phrases:} conformal mappings,
Sobolev spaces, elliptic equations.} \footnotetext{\textbf{2000
Mathematics Subject Classification:} 46E35, 35P15, 30C65.}

\section{Introduction }

Let $\Omega\subset\mathbb R^2$ be a bounded simply connected plane domain with a smooth boundary $\partial \Omega$. We consider the Neumann-Laplace spectral  problem (the free membrane problem)
\begin{gather}
-\Delta u=\lambda u\,\,\text{in}\,\,\Omega,\label{Neum01}\\
{\frac{\partial u}{\partial n}}\bigg|_{\partial\Omega}=0.\label{Neum02}
\end{gather}

The weak statement of the spectral problem of the Neumann-Laplace operator is as follows: A function
$u$ solves the previous problem iff $u\in W^{1,2}(\Omega)$ and 
\[
\iint\limits _{\Omega}\nabla u(x,y)\cdot\nabla v(x,y)~dxdy=\lambda\iint\limits _{\Omega}u(x,y)v(x,y)~dxdy
\]
for all $v\in W^{1,2}(\Omega)$. The weak statement of the Neumann-Laplace spectral problem is correct for any domain. 

\vskip 0.3cm

Let $\psi:\mathbb D\to\Omega$ be the Riemann conformal mapping of the unit disc $\mathbb D\subset\mathbb R^2$ onto $\Omega$ that is a simply connected plane domain. We shall say that $\Omega$ is a conformal $\alpha$-regular domain \cite{BGU1} if $\psi'\in L^{\alpha}(\mathbb D)$ for some $\alpha>2$. The degree $\alpha$ does not depends on choice of $\psi$  (by the Riemann Mapping Theorem) and depends on the hyperbolic metric on $\Omega$ only.

\vskip 0.3cm

The main result of the paper is

\vskip 0.3cm

{\bf Theorem A.}
{\it Let $\Omega\subset\mathbb R^2$ be a conformal $\alpha$-regular domain. Then the spectrum of Neumann-Laplace operator in $\Omega$ is discrete, can be written in the form
of a non-decreasing sequence
\[
0=\lambda_0[\Omega]<\lambda_{1}[\Omega]\leq\lambda_{2}[\Omega]\leq...\leq\lambda_{n}[\Omega]\leq...\,,
\]
and
\begin{equation}
\label{eq:est}
{1}/{\lambda_1[\Omega]} \leq \frac{4}{\sqrt[\alpha]{\pi^2}}\left(\frac{2\alpha-2}{\alpha-2}\right)^{\frac{2\alpha-2}{\alpha}}{\|\psi'\mid L^{\alpha}(\mathbb D)\|^2}
\end{equation}
where $\psi:\mathbb D\to\Omega$ is the Riemann conformal mapping of the unit disc $\mathbb D\subset\mathbb R^2$ onto $\Omega$. }

Note, that G.Polya in 1961 (\cite{P60}) obtained upper estimates for eigenvalues of Neumann-Laplace operator in so-called plane-covering domains. Namely, for the first eigenvalue:
$$
{\lambda_1[\Omega]}\leq 4\pi|\Omega|^{-1}.
$$

So, for the plane-covering conformal $\alpha$-regular domains we have the two-side estimate:
$$
\frac{\sqrt[\alpha]{\pi^2}}{4}\left(\frac{2\alpha-2}{\alpha-2}\right)^{\frac{2-2\alpha}{\alpha}}
{\frac{1}{\|\psi'\mid L^{\alpha}(\mathbb D)\|^2}}\leq {\lambda_1[\Omega]}\leq 4\pi\frac{1}{|\Omega|}.
$$

\vskip 0.3cm

The first non-trivial eigenvalue of the Neumann-Laplace operator is connected to the sharp constant in the isoperimetric 
inequalities \cite{M}.  Note, that lower estimates of the first non-trivial eigenvalue of the Neumann-Laplace operator in terms of isoperimetric constants were considered in \cite{BCT09, BCT15}

The Theorem A is based on existence of the universal weighted Poincar\'e-Sobolev inequality. Namely, in any simply connected plane domain with non-empty boundary we have: 

\begin{thm}
\label{thm:PoincareEnWeComp} Suppose that $\Omega\subset\mathbb R^2$ is a simply connected domain with non empty boundary,
$\varphi : \Omega\to\mathbb D$ is a conformal homeomorphism and $h(x,y)=J_{\varphi}(x,y)$ is the conformal weight. Then for every function 
$f\in W^{1,2}(\Omega,h,1)$, the weighted Poincar\'e-Sobolev inequality
\begin{equation}
\biggl(\iint\limits _{\Omega}|f(x,y)-f_{\Omega,h}|^{r}h(x,y)\, dxdy \biggr)^{\frac{1}{r}}\leq B_{r,2}[\Omega,h]\biggl(\iint\limits _{\Omega}
|\nabla f(x,y)|^{2}~dxdy\biggr)^{\frac{1}{2}}\label{eq:WPI}
\end{equation}
holds for every $r \in \left[ 1,\infty \right)$ with the exact constant 
$$
B_{r,2}[\Omega,h]=B_{r,2}[\mathbb D]\leq 2\pi^{\frac{2-r}{2r}}\left({(r+2)}/{2}\right)^{\frac{r+2}{2r}}
$$
 where  $B_{r,2}[\mathbb D]$ is the exact constant in the Poincar\'e inequality for the unit disk $\mathbb D$
$$
\biggl(\int\limits_{\mathbb D}|g(x,y)-g_{\mathbb D}|^r~dxdy\biggr)^{\frac{1}{r}}
\leq B_{r,2}[\mathbb D]\biggl(\int\limits_{\mathbb D}|\nabla g(x,y)|^2~dxdy\biggr)^{\frac{1}{2}}.
$$
\end{thm}

The paper is organized as follows:

In Section 2 we discuss the notion of conformal regular domains, formulate the Poincar\'e-Sobolev inequality for conformal regular domains and discuss its connection with the Neumann-Laplace operator. The main point is an estimate for the constant in this inequality. We also discuss connection between  the Poincar\'e-Sobolev inequality and composition operators on Sobolev spaces. Section 2 can be seen as an extension of the introduction. In Section 3 we prove main facts about composition operators in the conformal regular domains.  In Section 4 we prove the Poincar\'e-Sobolev inequality for conformal regular domains. In section 5 and 6 we apply results of Section 4 to lower estimates of the first eigenvalue for the Neumann-Laplace operator in the conformal regular domains. In Section 7 we compare the lower estimates with previous results available in the literature.

\section{The Neumann-Laplace problem in conformal regular domains}

Let $\Omega\subset\mathbb R^2$ be a simply connected plane domain of finite area and $\psi:\mathbb D\to\Omega$ be a conformal mapping. Then
$$
\iint\limits_{\mathbb D} |\psi'(u,v)|^2~dudv= \iint\limits_{\mathbb D} J_{\psi}(u,v)~dudv=|\Omega|<\infty.
$$
Integrability of the derivative in the power $\alpha>2$ is impossible without additional assumptions on the geometry of $\Omega$. We proved in \cite{GU4} that the integrability in the power $\alpha>2$ leads to finiteness of the geodesic diameter of $\Omega$ and as result it is a bounded domain. 

A domain $\Omega$ is a conformal regular domain if it is an $\alpha$-regular domain for some $\alpha>2$. Note that any $C^2$-smooth simply connected bounded domain domain is $\infty$-regular (see, for example, \cite{Kr}).

\vskip0.3cm

The notion of conformal regular domains was introduced in \cite{BGU1} and was applied to the stability problem for eigenvalues of the Dirichlet-Laplace operator.  It does not depends on choice
of a conformal mapping $\psi:{\mathbb D}\to\Omega$ and can be reformulated
in terms of the hyperbolic metrics \cite{BGU1}. Namely
\[
\iint\limits _{\mathbb D}|\psi'(u,v)|^{\alpha}~dudv=
\iint\limits _{\mathbb D}\left(\frac{\lambda_{\mathbb D}(u,v)}{\lambda_{\Omega}(\psi(u,v))}\right)^{\alpha}~dudv
\]
 when $\lambda_{\mathbb D}$ and $\lambda_{\Omega}$ are hyperbolic
metrics in $\mathbb D$ and $\Omega$ \cite{BM}.

\vskip0.3cm

Note that a boundary $\partial \Omega$ of a conformal regular domain can have any Hausdorff dimension between one and two, but can not be equal two \cite{HKN}.

\vskip0.3cm

In Section 4 we prove the following Poincar\'e-Sobolev inequality for conformal regular domains

\begin{thm}
\label{thm:PoincareEnComp} Suppose that $\Omega\subset\mathbb R^2$ is a conformal $\alpha$-regular domain.
Then for every function 
$f\in W^{1,2}(\Omega)$, the inequality 
\begin{equation}
\inf\limits_{c\in\mathbb R}\biggl(\iint\limits _{\Omega}|f(x,y)-c|^{s}\, dxdy \biggr)^{\frac{1}{s}}\leq B_{s,2}[\Omega]\biggl(\iint\limits _{\Omega}
|\nabla f(x,y)|^{2}~dxdy\biggr)^{\frac{1}{2}}\label{eq:NonWPI}
\end{equation}
holds with the constant 
$$
B_{s,2}[\Omega]\leq \|J_{\varphi^{-1}}| L^{\frac{r}{r-s}}(\mathbb D)\|^{\frac{1}{s}}B_{r,2}[\Omega,h]
\leq 2\pi^{\frac{2-r}{2r}}\left({(r+2)}/{2}\right)^{\frac{r+2}{2r}}\cdot\|J_{\varphi^{-1}}| L^{\frac{r}{r-s}}(\mathbb D)\|^{\frac{1}{s}} 
$$
for any $s\in[1,\infty)$, where $r={\alpha s}/{(\alpha-2)}$.
\end{thm}

\begin{rem} The conformal regular domains have an equivalent description in the terms of the $\beta $-(quasi)\-hy\-per\-bolic boundary condition \cite{BP,KOT}. 
In \cite{KOT} it was proved (without estimates of constants) that domains with $\beta$-(quasi)hyperbolic boundary conditions support the $(s,p)$-Poincar\'e-Sobolev inequalities for  $p$ that depends on $\beta$.
\end{rem}

Existence of the Poincar\'e-Sobolev inequality is an essential property of the conformal regular domains. In \cite{GU4} we proved but did not formulated the following fact about conformal regular domains:

\begin{thm}
Let a simply connected domain $\Omega\subset\mathbb R^2$ of finite measure does not support the (s,2)-Poincar\'e-Sobolev inequality
$$
\biggl(\int\limits_{\Omega}|f(x,y)-f_{\Omega}|^{s}~dxdy\biggr)^{\frac{1}{s}}\leq B_{s,2}[\Omega]\biggl(\int\limits_{\Omega}|\nabla f(x,y)|^2~dxdy\biggr)^{\frac{1}{2}}
$$
for some $s\geq 2$. Then $\Omega$ is not a conformal regular domain.
\end{thm}

\begin{rem} Conformal regular domains domains allows narrow gaps which can destroy the John condition \cite{HMV}.
\end{rem}

It is well known that solvability of the Neumann-Laplace problem and its spectrum
discreteness depends on regularity of $\Omega$ (see, for example, \cite{M}). 

In the present work we suggest a new method for a study of the Poincar\'e-Sobolev inequality in the conformal regular domains. This method is based on the composition operators theory on Sobolev spaces.  It allows us to estimate constants in the Poincar\'e-Sobolev inequalities. 

As an application we study the eigenvalues problem for the Neumann-Laplace operator.

A detailed survey of this eigenvalue problem can be found in \cite{NKP} (for example). A global lower bound of the non-trivial  first eigenvalue  
$\lambda_1[\Omega]$ for convex domains was obtained  in \cite{PW}. We obtain a global lower bound for the first eigenvalue  
$\lambda_1[\Omega]$ in conformal $\alpha$-regular domains, which are not necessary convex domains. 

The suggested method is based on the composition
operators theory \cite{U1,VU1} and its applications to the Sobolev
type embedding theorems \cite{GGu,GU}.
The following diagram illustrate this idea:

\[\begin{array}{rcl}
W^{1,p}(\Omega) & \stackrel{(\varphi^{-1})^*}{\longrightarrow} & W^{1,q}(\mathbb D) \\[2mm]
\multicolumn{1}{c}{\downarrow} & & \multicolumn{1}{c}{\downarrow} \\[1mm]
L^s(\Omega) & \stackrel{\varphi^*}{\longleftarrow} & L^r(\mathbb D)
\end{array}\]

Here the operator $\varphi^{\ast}$ defined by the composition rule $\varphi^{\ast}(f)=f\circ\varphi$ is a bounded composition operator on Lebesgue spaces induced by a homeomorphism $\varphi$ of $\Omega$ and $\mathbb D$ and the operator $(\varphi^{-1})^{\ast}$ defined by the composition rule $(\varphi^{-1})^{\ast}(f)=f\circ\varphi^{-1}$ is a bounded composition operator on Sobolev spaces.

\begin{rem} In the recent works we studied 
composition operators on Sobolev spaces defined on planar domains
in connection with the conformal mappings theory \cite{GU1}. This
connection leads to weighted Sobolev embeddings \cite{GU2,GU3} with
the universal conformal weights. Another application of conformal
composition operators was given in \cite{BGU1} where the spectral
stability problem for conformal regular domains was considered.
\end{rem}

\section{Composition Operators}

\subsection{Composition Operators on Lebesgue Spaces}

For any domain $\Omega \subset \mathbb{R}^{2}$ and any $1\leq p<\infty$
we consider the Lebesgue space 
\[
L^{p}(\Omega):=\left\{ f:\Omega \to R:\|f\mid{L^{p}(\Omega)}\|:=\left(\iint\limits _{\Omega}|f(x,y)|^{p}~dxdy\right)^{1/p}<\infty\right\} .
\]
 
The following theorem  about composition operator on Lebesgue spaces is well known (see, for example \cite{VU1}): 

\begin{thm}
\label{thm:LpLq} A diffeomorphism $\varphi:\Omega\to\Omega'$ between two plane domains $\Omega$ and $\Omega'$ induces a bounded composition operator 
\[
\varphi^{\ast}:L^{r}(\Omega')\to L^{s}(\Omega),\,\,\,1\leq s\leq r<\infty,
\]
(by the chain rule $\varphi^{\ast}(f):=f \circ \varphi$) if and only if 
\begin{gather}
\biggl(\iint\limits _{\Omega'}\left(J_{\varphi^{-1}}(u,v)\right)^{\frac{r}{r-s}}~dudv\biggl)^{\frac{r-s}{rs}}=K<\infty, \,\,\,1\leq s<r<\infty,\nonumber\\
\left(\operatorname{esssup}_{(u,v)\in\Omega'}J_{\varphi^{-1}}(u,v)\right)^{\frac{1}{s}}=K<\infty, \,\,\,1\leq s=r<\infty.
\nonumber
\end{gather}
 The norm of the composition operator $\|\varphi^{\ast}\|=K$.
\end{thm}

\subsection{Composition Operators on Sobolev Spaces}

We define the Sobolev space $W^{1,p}(\Omega)$, $1\leq p<\infty$,
as a Banach space of locally integrable weakly differentiable functions
$f:\Omega\to\mathbb{R}$ equipped with the following norm: 
\[
\|f\mid W^{1,p}(\Omega)\|=\biggr(\iint\limits _{\Omega}|f(x,y)|^{p}\, dxdy\biggr)^{\frac{1}{p}}+\biggr(\iint\limits _{\Omega}|\nabla f(x,y)|^{p}\, dxdy\biggr)^{\frac{1}{p}}.
\]

We define also the homogeneous seminormed Sobolev space $L^{1,p}(\Omega)$
of locally integrable weakly differentiable functions $f:\Omega\to\mathbb{R}$ equipped
with the following seminorm: 
\[
\|f\mid L^{1,p}(\Omega)\|=\biggr(\iint\limits _{\Omega}|\nabla f(x,y)|^{p}\, dxdy\biggr)^{\frac{1}{p}}.
\]

Recall that the embedding operator $i:L^{1,p}(\Omega)\to L_{\loc}^{1}(\Omega)$
is continuous.

\begin{rem}
By the standard definition functions of $L^{1,p}(\Omega)$ are defined only up to a set of measure zero, but they can be redefined quasi-everywhere i.~e. up to a set of $p$-capacity zero (see, for example \cite{HKM,M}).
\end{rem}

Let $\Omega$ and $\Omega'$ be domains in $\mathbb R^2$. We say that
a diffeomorphism $\varphi:\Omega\to\Omega'$ induces a bounded composition
operator 
\[
\varphi^{\ast}:L^{1,p}(\Omega')\to L^{1,q}(\Omega),\,\,\,1\leq q\leq p\leq\infty,
\]
 by the composition rule $\varphi^{\ast}(f)=f\circ\varphi$, if for
any $f\in L^{1,p}(\Omega')$ the composition $\varphi^{\ast}(f)\in L^{1,q}(\Omega)$
and there exists a constant $K<\infty$ such that 
\[
\|\varphi^{\ast}(f)\mid L^{1,q}(\Omega)\|\leq K\|f\mid L^{1,p}(\Omega')\|.
\]

The main result of \cite{U1, VU1} gives an analytic description of composition
operators on Sobolev spaces $L^{1,p}$. We reproduce it here for diffeomorphisms.

\begin{thm}
\label{CompTh} \cite{U1} A diffeomorphism $\varphi:\Omega\to\Omega'$
between two domains $\Omega$ and $\Omega'$ induces a bounded composition
operator 
\[
\varphi^{\ast}:L^{1,p}(\Omega')\to L^{1,q}(\Omega),\,\,\,1\leq q<p<\infty,
\]
 if and only if
\[
K_{p,q}(\varphi;\Omega)=\biggl(\iint\limits _{\Omega}\biggl(\frac{|\varphi'(x,y)|^{p}}{|J_{\varphi}(x,y)|}\biggr)^{\frac{q}{p-q}}~dxdy\biggr)^{\frac{p-q}{pq}}<\infty.
\]
The norm of the composition operator $\|\varphi^{\ast}\|\leq K_{p,q}(\varphi;\Omega)$
\end{thm}

\begin{defn}
We call a bounded domain $\Omega\subset\mathbb{C}$
as $(r,q)$-Poincar\'e domain, $1\leq q,r\leq\infty$, if the following Poincar\'e-Sobolev inequality
$$
\inf\limits _{c\in\mathbb{R}}\|g-c\mid L^{r}(\Omega)\|
\leq B_{r,q}[\Omega] \|g\mid L^{1,q}(\Omega)\|
$$
holds for any $g\in L^{1,q}(\Omega)$  with the constant $B_{r,q}[\Omega]<\infty$. The unit disc $\mathbb D\subset\mathbb R^2$ is an example of the $(r,2)$-embedding domain for all $r\geq 1$. 
\end{defn}

The following theorem gives a characterization of composition operators in the classical Sobolev spaces $W^{1,p}$. This theorem was proved, but did
not formulated in \cite{GGu,GU}. For readers convenience we reproduce here the proof. 

\begin{thm}
\label{thm:boundWW} Let $\Omega\subset\mathbb R^n$ be an $(r,q)$-Poincar\'e domain for some $1\leq q\leq r\leq\infty$ and a domain $\Omega'$ has finite measure.
Suppose that a diffeomorphism $\varphi:\Omega\to\Omega'$ induces
a bounded composition operator 
\[
\varphi^{\ast}:L^{1,p}(\Omega')\to L^{1,q}(\Omega),\,\,\,1\leq q\leq p<\infty,
\]
 and the inverse diffeomorphism $\varphi^{-1}:\Omega'\to\Omega$ induces
a bounded composition operator 
\[
(\varphi^{-1})^{\ast}:L^{r}(\Omega)\to L^{s}(\Omega'),\,\,\,1\leq s\leq r<\infty,
\]
 for some $p\leq s\leq r$.

 Then $\varphi:\Omega\to\Omega'$ induces
a bounded composition operator 
\[
\varphi^{\ast}:W^{1,p}(\Omega')\to W^{1,q}(\Omega),\,\,\,1\leq q\leq p<\infty.
\]
 \end{thm}
\begin{proof}
 Let $f \in W^{1,p}(\Omega)$ and $g=\varphi^{\ast}(f)$. 
Because $p\leq r$ and the composition operator $(\varphi^{-1})^{\ast}:L^{r}(\Omega)\to L^{s}(\Omega')$
is bounded the following inequality 
\[
\|(\varphi^{-1})^{\ast}g\mid L^{s}(\Omega')\|\leq A_{r,s}(\Omega)\|g\mid L^{r}(\Omega)\|
\]
 is correct for a positive constant $A_{r,s}(\Omega)$. 

Since domain $\Omega$ is a $(r,q)$-Poincar\'e domain 
$$
\inf\limits _{c\in\mathbb{R}}\|g-c\mid L^{r}(\Omega)\|
\leq B_{r,q}[\Omega]\|g\mid L^{1,q}(\Omega)\|.
$$
and the composition operator 
\[
\varphi^{\ast}:L^{1,p}(\Omega')\to L^{1,q}(\Omega)
\]
is bounded we obtain finally the following
inequalities 
\begin{multline*}
\inf\limits _{c\in\mathbb{R}}\|f-c\mid L^{s}(\Omega')\|\leq A_{r,s}(\Omega)\inf\limits _{c\in\mathbb{R}}\|g-c\mid L^{r}(\Omega)\|\\
\leq A_{r,s}(\Omega)B_{r,q}[\Omega]\|g\mid L^{1,q}(\Omega)\|\leq A_{r,s}(\Omega)K_{p,q}(\Omega)B_{r,q}[\Omega]\|f\mid L^{1,p}(\Omega')\|
\end{multline*}
 hold. Here  $K_{p,q}(\Omega)$ is the upper bound of the norm of the composition operator $\varphi^{\ast}:L^{1,p}(\Omega')\to L^{1,q}(\Omega)$.

The H\"older inequality implies the following estimate 
\begin{multline*}
|c|=|\Omega'|^{-\frac{1}{p}}\|c\mid L^{p}(\Omega')\|\leq|\Omega'|^{-\frac{1}{p}}\bigl(\|f\mid L^{p}(\Omega')\|+\|f-c\mid L^{p}(\Omega')\|\bigr)\\
\leq|\Omega'|^{-\frac{1}{p}}\|f\mid L^{p}(\Omega')\|+|\Omega'|^{-\frac{1}{s}}\|f-c\mid L^{s}(\Omega')\|.
\end{multline*}

Because $q\leq r$ we have 
\begin{multline*}
\|g\mid L^{q}(\Omega)\|\leq\|c\mid L^{q}(\Omega)\|+\|g-c\mid L^{q}(\Omega)\|\leq|c||\Omega|^{\frac{1}{q}}+|\Omega|^{\frac{r-q}{r}}\|g-c\mid L^{r}(\Omega)\|\\
\leq\biggl(|\Omega'|^{-\frac{1}{p}}\|f\mid L^{p}(\Omega')\|+|\Omega'|^{-\frac{1}{s}}\|f-c\mid L_{s}(\Omega')\|\biggr)|\Omega|^{\frac{1}{q}}+|\Omega|^{\frac{r-q}{r}}\|g-c\mid L^{r}(\Omega)\|.
\end{multline*}

Combining previous inequalities we obtain finally
\begin{multline*}
\|g\mid L^{q}(\Omega)\|\leq|\Omega|^{\frac{1}{q}}|\Omega'|^{-\frac{1}{p}}\|f\mid L^{p}(\Omega')\|\\
+A_{r,s}(\Omega)K_{p,q}(\Omega)B_{r,q}[\Omega]|\Omega|^{\frac{1}{q}}|\Omega'|^{-\frac{1}{p}}\|f\mid L^{1,p}(\Omega')\|\\
+K_{p,q}(\Omega)B_{r,q}[\Omega]|\Omega|^{\frac{r-q}{r}}\|f\mid L^{1,p}(\Omega)\|.
\end{multline*}
 Therefore the composition operator 
\[
\varphi^{\ast}:W^{1,p}(\Omega')\to W^{1,q}(\Omega)
\]
 is bounded. 
\end{proof}

\section{Poincar\'e-Sobolev inequalities for functions of $W^{1,2}(\Omega)$}

\subsection{Weighted Lebesgue spaces}

We follow \cite{HKM} for notation and basic facts about weighted
Lebesgue spaces.

Let $\Omega\subset\mathbb{R}^{2}$ be a domain and let $v:\Omega\to \mathbb R$
be a locally integrable almost everywhere positive real valued function in $\Omega$
( i.e $v>0$ almost everywhere). Then a Radon measure $\nu$ is
canonically associated with the weight function $v$:
\[
\nu(E):=\iint_{E}v(x,y)~dxdy.
\]

By the local integrability of $v$, the measure $\nu$ and the Lebesgue
measure are absolutely continuous with respect one to another:
\[
d\nu=v(x,y)dxdy.  
\]
 In what follows, the weight $v$ and the measure $\nu$ will be identified.
The sets of measure zero are the same for the Lebesgue measure and for the measure $\nu$.
It means that we do not need to specify what convergence almost everywhere is. 

Denote by 
$$
\mathcal{V}(\Omega):=\{v\in L^1_{ \loc}(\Omega):v>0\,\, \text{a.~e. on}\,\, \Omega \}
$$ the set of all such weights. Here $L^1_{\loc}(\Omega)$
is the space of locally integrable functions in $\Omega$.

For $1\leq p<\infty$ and $v\in\mathcal{V}(\Omega)$, consider the
weighted Lebesgue space 
\[
L^{p}(\Omega,v):=\left\{ f:\Omega\to R:\| f\mid {L^{p}(\Omega,v)}\|:=\left(\iint_{\Omega}|f(x,y)|^{p}v(x,y)~dxdy\right)^{1/p}<\infty \right\}.
\]
It is a Banach space for the norm $\|f\mid{L^{p}(\Omega,v)}\|$.

The space $L^{p}(\Omega,v)$ may fail to embed into $L^1_{\loc}(\Omega)$. 

\begin{prop}
\cite{HKM}\label{WL} If $v^{\frac{1}{1-p}}\in L^1_{\loc}(\Omega)$
and $1< p<\infty$ then the embedding operator $i:L^{p}(\Omega,v)\to L^1_{\loc}(\Omega)$
is continuous. 

If $v^{-1}\in L^{\infty}_{\loc}(\Omega)$ then the embedding operator
$i:L^{1}(\Omega,v)\to L^1_{\loc}(\Omega)$ is continuous. 
\end{prop}
For $1<p<\infty$, we put
\[
\mathit{\mathcal{V}_{p}(\Omega):=\left\{ v\in\mathcal{V}(\Omega):v^{\frac{1}{1-p}}\in L^1_{\loc}(\Omega)\right\} }
\]
and for $p=1$, 
\[
\mathit{\mathcal{V}_{1}(\Omega):=\left\{ v\in\mathcal{V}(\Omega):v^{-1}\in L^{\infty}_{\loc}(\Omega)\right\} }.
\]

\begin{cor}
If a weight $v$ is continuous and positive then $i:L^{p}(\Omega,v)\to L^1_{\loc}(\Omega)$
is continuous.
\end{cor}

This follows immediately from Proposition \ref{WL} because a continuous and
positive weight belongs to $\mathcal{V}_{p}(\Omega)$ and also to
$\mathcal{V}_{1}(\Omega)$.

\subsection{Weighted Poincar\'e-Sobolev inequalities}

Let $\varphi :\Omega\to\Omega'$ be a conformal homeomorphism. The following fact is well-known:
\begin{lem}
\label{lem:isometry}Let $\Omega$ and $\Omega'$ be two plane domains.
Any conformal homeomorphism $w=\varphi(z):\Omega\to\Omega'$ induces
an isometry of spaces $L^{1,2}(\Omega')$ and $L^{1,2}(\Omega)$.\end{lem}
\begin{proof}
Let $f\in L^{1,2}(\Omega')$ be a smooth function. Then the smooth
function $g=f\circ\varphi$ belongs to $L^{1,2}(\Omega)$ because
\begin{multline*}
\|\nabla g\mid L^{2}({\Omega})\|=\biggl(\iint\limits _{{\Omega}}|\nabla(f\circ\varphi(x,y))|^{2}~dxdy\biggr)^{\frac{1}{2}}\\
=\biggl(\iint\limits _{{\Omega}}|\nabla f|^{2}(\varphi(x,y))|\varphi'(x,y))|^{2}~dxdy\biggr)^{\frac{1}{2}}
=\biggl(\iint\limits _{{\Omega}}|\nabla f|^{2}(\varphi(x,y))J_{\varphi}(x,y)~dxdy\biggr)^{\frac{1}{2}}\\
=\biggl(\iint\limits _{\Omega'}|\nabla f|^{2}(u,v)~dudv\biggr)=\|\nabla f\mid L^{2}(\Omega')\|.
\end{multline*}

We used the equality: $|\varphi'(x,y))|^{2}=J_{\varphi}(x,y)$ that is correct for any conformal homeomorphism.

Approximating an arbitrary function $f\in L^{1,2}(\Omega')$
by smooth functions, we obtain an isometry between $L^{1,2}(\Omega')$
and $L^{1,2}(\Omega)$. 
\end{proof}

We define the weighted Sobolev space $W^{1,p}(\Omega,h,1)$, $1\leq p<\infty$,
as the normed space of all locally integrable weakly differentiable functions
$f:\Omega\to\mathbb{R}$ with the finite norm given by
$$
\|f\mid W^{1,p}(\Omega,h,1)\|=\|f\mid L^p(\Omega,h)\|+\|\nabla f\mid L^p(\Omega)\|.
$$

The existence of the Poincar\'e-Sobolev inequalities depends on a conformal (hyperbolic)  geometry of $\Omega$. For any conformal homeomorphism $\varphi: \Omega\to\mathbb D$ define the conformal (hyperbolic) weight $h(x,y):=J_{\varphi}(x,y)$. 

We denote
\begin{multline}
f_{{\Omega}, h}:=\frac{1}{m_h(\Omega)}\iint\limits_{\mathbb D}f(z)h(z)~dxdy=g_{\mathbb D}=\frac{1}{|\mathbb D|}\iint\limits_{\mathbb D} g(w)~dudv,\\
f(z)=g(\varphi(z)), \,\,\, w=\varphi(z).
\nonumber
\end{multline}
Here
$$
m_h(\Omega)=\iint\limits_{\Omega}h(z)~dxdy=\iint\limits_{\Omega}J_{\varphi}(z)~dxdy=\iint\limits_{\mathbb D}~dudv=|\mathbb D|.
$$

The following "universal" weighted Poincar\'e-Sobolev inequality is correct for any simply connected plane domain with non-empty boundary.

\begin{thm}
\label{thm:PoinW12} 
Let $\Omega$ be a simply connected plane domain with non-empty boundary. Then
for any function $f\in W^{1,2}(\Omega,h,1)$ the weighted Poincar\'e-Sobolev inequality 
$$
\biggl(\iint\limits_{\Omega}|f(x,y)-f_{\Omega,h}|^r h(x,y)~dxdy\biggr)^{\frac{1}{r}}\leq B_{r,2}[{\Omega},h]\biggl(\iint\limits_{\Omega}|\nabla f(x,y)|^2~dxdy\biggr)^{\frac{1}{2}}
$$
holds for any $r\geq 1$
with the exact constant $B_{r,2}[\Omega,h]=B_{r,2}[\mathbb D]$ where  $B_{r,2}[\mathbb D]$ is the exact constant of the Poincar\'e inequality in the unit disk 
$$
\biggl(\int\limits_{\mathbb D}|g(x,y)-g_{\mathbb D}|^r~dxdy\biggr)^{\frac{1}{r}}
\leq B_{r,2}[\mathbb D]\biggl(\int\limits_{\mathbb D}|\nabla g(x,y)|^2~dxdy\biggr)^{\frac{1}{2}}.
$$
\end{thm}

\begin{proof}
Let $r\geq 1$. By the Riemann Mapping Theorem there exists a conformal homeomorphism $\varphi: \Omega\to\mathbb D$. Using the change of variable formula for conformal homeomorphism,  the Poincar\'e-Sobolev inequality in the unit disc $\mathbb D\subset\mathbb R^2$ and Lemma \ref{lem:isometry} we get

\begin{multline}
\biggl(\iint\limits_{\Omega}|f(x,y)-f_{\Omega,h}|^r h(x,y)~dxdy\biggr)^{\frac{1}{r}}=\biggl(\iint\limits_{\Omega}|f(x,y)-g_{\mathbb D}|^r h(x,y)~dxdy\biggr)^{\frac{1}{r}} 
\\
=\biggl(\iint\limits_{\Omega}|f(x,y)-g_{\mathbb D}|^r J_{\varphi}(x,y)~dxdy\biggr)^{\frac{1}{r}}=
\biggl(\iint\limits_{\mathbb D}|g(u,v)-g_{\mathbb D}|^r ~dudv\biggr)^{\frac{1}{r}}\\
\leq B_{r,2}[\mathbb D]\left(\iint\limits_{\mathbb D}|\nabla g(u,v)|^2~dudv\right)^{\frac{1}{2}}
= B_{r,2}[\mathbb D]\left(\iint\limits_{\Omega}|\nabla f(x,y)|^2~dxdy\right)^{\frac{1}{2}}.
\nonumber
\end{multline}
for any function $f\in W^{1,2}(\Omega,h,1)$.
\end{proof}

Let us estimate of $B_{r,2}[\mathbb D]$ using the following $n$-dimensional inequalities \cite{GT}. For any $\mu\in(0,1)$ and any domain $\Omega\subset\mathbb R^n$ define the operator $V_{\mu}$ acting on $L^1(\Omega)$ by the expression
$$
\left(V_{\mu}f\right)(x)=\int\limits_{\Omega}|x-y|^{n(\mu-1)}f(y)~dy.
$$
Here $x=(x_1,...,x_n)$, $y=(y_1,...,y_n)$ and $dy=dy_1...dy_n$.

\begin{lem}\cite{GT}
\label{lem:GT1}
The operator $V_{\mu}$ maps $L^p(\Omega)$ continuously into $L^q(\Omega)$ for any $q$, $1\leq q\leq\infty$ satisfying
$$
0\leq \delta=\delta(p,q)=p^{-1}-q^{-1}<\mu.
$$
Furthermore, for any $f\in L^p(\Omega)$,
$$
\|V_{\mu}f\mid L^q(\Omega)\|\leq \left(\frac{1-\delta}{\mu-\delta}\right)^{1-\delta}\omega_n^{1-\mu}|\Omega|^{\mu-\delta}\|f\mid L^p(\Omega)\|.
$$
Here $\omega_n=\frac{2\pi^{n/2}}{n\Gamma(n/2)}$ is the volume of the unit ball in $\mathbb R^n$.
\end{lem}

In the convex domains there are the following point-wise estimates

\begin{lem}\cite{GT}
\label{GT2}
Let $\Omega$ be a convex domain and $f\in W^{1,p}(\Omega)$. Then
$$
|f(x)-f_{\Omega}|\leq \frac{d^n}{n|\Omega|}\int\limits_{\Omega}|x-y|^{1-n}|\nabla f(y)|~dy\,\,\,\text{a.~e. in}\,\,\,\Omega
$$
where $d$ is the diameter of $\Omega$.
\end{lem}

From these two lemmas follows

\begin{prop}
\label{EstCon}
Let $\Omega$ be a convex domain and $f\in W^{1,p}(\Omega)$. Then
\begin{multline}
\left(\int\limits_{\Omega}|f(x)-f_{\Omega}|^q~dx\right)^{\frac{1}{q}}
\leq
\frac{d^n}{n|\Omega|}\left(\frac{1-\frac{1}{p}+\frac{1}{q}}{\frac{1}{n}-\frac{1}{p}+\frac{1}{q}}\right)^{1-\frac{1}{p}+\frac{1}{q}}\omega_n^{1-\frac{1}{n}}|\Omega|^{\frac{1}{n}-\frac{1}{p}+\frac{1}{q}}\||\nabla f| \mid L^p(\Omega)\|.
\nonumber
\end{multline}
\end{prop}

\begin{proof}
We take $\mu=1/n$. Then for function $f\in W^{1,p}(\Omega)$ we have
\begin{multline}
\left(\int\limits_{\Omega}|f(x)-f_{\Omega}|^q~dx\right)^{\frac{1}{q}}\\\leq 
\frac{d^n}{n|\Omega|}\left(\int\limits_{\Omega}\left|\int\limits_{\Omega}|x-y|^{1-n}|\nabla f(y)|~dy\right|^q~dx\right)^{\frac{1}{q}}
=
\frac{d^n}{n|\Omega|}\|V_{\frac{1}{n}} |\nabla f| \mid L^q(\Omega)\|\\
\leq
\frac{d^n}{n|\Omega|}\left(\frac{1-1/p+1/q}{1/n-1/p+1/q}\right)^{1-1/p+1/q}\omega_n^{1-1/n}|\Omega|^{1/n-1/p+1/q}\||\nabla f \mid L^p(\Omega)\|.
\nonumber
\end{multline}
\end{proof}

Proposition \ref{EstCon} give us an upper the estimate of the constant in the weighted Poincar\'e-Sobolev inequality in any simply connected plane domain with non-empty boundary.
$$
B_{r,2}[\Omega,h]=B_{r,2}[\mathbb D]\leq 2\pi^{\frac{2-r}{2r}}\left({(r+2)}/{2}\right)^{\frac{r+2}{2r}}.
$$

\vskip 1cm

We are ready to prove the main technical result of this work:

\begin{thm}
\label{thm:comembd} 
If $\Omega$ is a conformal $\alpha$-regular domain then:

1) The embedding operator
\[
i:W^{1,2}(\Omega)\hookrightarrow L^{s}(\Omega),\,\,\,
\]
is compact for any $s\geq 1$.

2) For any function $f\in W^{1,2}(\Omega)$ the Poincar\'e-Sobolev inequality
$$
\inf\limits_{c\in \mathbb R}\biggl(\int\limits_{\Omega}|f(x,y)-f_{\Omega}|^s~dxdy\biggr)^{\frac{1}{s}}\leq B_{s,2}[\Omega]\biggl(\int\limits_{\Omega}|\nabla f(x,y)|^2~dxdy\biggr)^{\frac{1}{2}}
$$
holds for any $s\geq 1$.

3) The following estimate is correct $B_{s,2}[\Omega]\leq B_{r,2}[\mathbb D]\cdot \|\psi'|L^{\alpha}(\mathbb D)\|^{\frac{2}{s}}$. Here $B_{r,2}[\mathbb D]$ is the exact constant for the Poincar\'e inequality in the unit disc, $r=\alpha s/(\alpha-2)$. 

\end{thm}

\begin{proof}
Let $s\geq1$. Since $\Omega$ is a conformal $\alpha$-regular domain then
 for any conformal homeomorphism $\varphi:\Omega\to\mathbb D$ its inverse conformal homeomorphism $\psi=\varphi^{-1}$ satisfy to the following condition of $\alpha$-regularity
\[
\iint\limits _{\mathbb D}|\psi'(u,v)|^{\alpha}~dudv=
\iint\limits _{\mathbb D }|J_{\varphi^{-1}}(u,v)|^{\alpha/2}~dudv<\infty.
\]
For the unit disc $\mathbb D$ the embedding
operator 
\[
i_{D}:W^{1,2}(\mathbb D)\hookrightarrow L^{r}(\mathbb D)
\]
 is compact (see, for example, \cite{M}) for any $r\geq1$.
 
 By Theorem~\ref{thm:LpLq} the composition
operator 
\[
\varphi^{\ast}:L^{r}(\mathbb D)\to L^{s}(\Omega)
\]
 is bounded if 
\[
\iint\limits_{\mathbb D}|J_{\varphi^{-1}}(u,v)|^{\frac{r}{r-s}}~dudv<\infty.
\]
Because $\Omega$ is a conformal $\alpha$-regular domain this condition holds for $\frac{r}{r-s}=\alpha/2$ i.e for $r=s\alpha/(\alpha-2)$.

Since a conformal homeomorphism $\varphi^{-1}$ induces a bounded composition
operator 
\[
(\varphi^{-1})^{\ast}:L^{1,2}(\Omega)\to L^{1,2}(\mathbb D)
\]
 then by Theorem~\ref{thm:boundWW} the composition
operator 
\[
(\varphi^{-1})^{\ast}:W^{1,2}(\Omega)\to W^{1,2}(\mathbb D)
\]
 is bounded. 

Therefore the imbedding operator 
\[
i_{\Omega}:W^{1,2}(\Omega)\hookrightarrow L^{s}(\Omega)
\]
is compact as a composition of bounded composition operators $\varphi^{\ast}$,
$(\varphi^{-1})^{\ast}$ and the compact embedding operator $i_{D}$
\[
i_{D}:W^{1,2}(D)\hookrightarrow L^{r}(\mathbb D).
\]

The first part of this theorem is proved. 

For any function  $f\in W^{1,2}(\Omega)$ and $r=s\alpha/(\alpha-2)$ the following inequalities are correct:

\begin{multline}
\inf\limits_{c\in\mathbb R}\biggl(\iint\limits_{\Omega}|f(x,y)-c|^s~dxdy\biggr)^{\frac{1}{s}}\leq \biggl(\iint\limits_{\Omega}|f(x,y)-f_{\Omega,h}|^s~dxdy\biggr)^{\frac{1}{s}}\\
\leq\left(\iint\limits_{\Omega}|J_{\varphi}(x,y)|^{-\frac{s}{r-s}}~dxdy\right)^{\frac{r-s}{rs}}\cdot
\biggl(\iint\limits_{\Omega}|f(x,y)-f_{\Omega,h}|^r h(x,y)~dxdy\biggr)^{\frac{1}{r}}\\
=
\left(\iint\limits_{\mathbb D}|J_{\varphi^{-1}}(u,v)|^{\frac{r}{r-s}}~dudv\right)^{\frac{r-s}{rs}}\cdot
\biggl(\iint\limits_{\Omega}|f(x,y)-f_{\Omega,h}|^r h(x,y)~dxdy\biggr)^{\frac{1}{r}} \\
=
\left(\iint\limits_{\mathbb D}|J_{\varphi^{-1}}(u,v)|^{\frac{\alpha}{2}}~dudv\right)^{\frac{2}{\alpha s}}\cdot
\biggl(\iint\limits_{\Omega}|f(x,y)-f_{\Omega,h}|^r h(x,y)~dxdy\biggr)^{\frac{1}{r}}.
\end{multline}

Using Theorem \ref{thm:PoinW12} we obtain

$$
\inf\limits_{c\in\mathbb R}\biggl(\iint\limits_{\Omega}|f(x,y)-f_{\Omega}|^s~dxdy\biggr)^{\frac{1}{s}}\leq B_{r,2}[\mathbb D]\cdot \|\psi'|L^{\alpha}(\mathbb D)\|^{\frac{2}{s}}\biggl(\iint\limits_{\Omega}|\nabla f(x,y)|^2~dxdy\biggr)^{\frac{1}{2}}.
$$
\end{proof}

For $\alpha=\infty$ the following analog of the previous theorem is correct:

\begin{thm}
\label{thm:infembd} 
If $\Omega$ is a conformal $\infty$-regular domain then:

1) The embedding operator
\[
i:W^{1,2}(\Omega)\hookrightarrow L^{2}(\Omega),\,\,\,
\]
is compact.

2) For any function $f\in W^{1,2}(\Omega)$ the Poincar\'e-Sobolev inequality
$$
\biggl(\iint\limits_{\Omega}|f(x,y)-f_{\Omega}|^2~dxdy\biggr)^{\frac{1}{2}}\leq B_{2,2}[\Omega]\biggl(\iint\limits_{\Omega}|\nabla f(x,y)|^2~dxdy\biggr)^{\frac{1}{2}}
$$
holds.

3) The following estimate is correct $B_{2,2}[\Omega]\leq B_{2,2}[\mathbb D]\cdot \|\psi'|L^{\infty}(\mathbb D)\|$. 
Here $B_{2,2}[\mathbb D]=1/\sqrt{\lambda_1[\Omega]}$ is the exact constant for the Poincar\'e inequality in the unit disk. 
\end{thm}

\begin{proof}
Since $\Omega$ is a conformal $\infty$-regular domain then
 for any conformal homeomorphism $\varphi:\Omega\to\mathbb D$ its inverse conformal homeomorphism $\psi=\varphi^{-1}$ satisfy to the following condition:
 
\[ 
\|\psi' \mid L^{\infty} (\mathbb D)\|^2=\|J_{\varphi^{-1}} \mid L^{\infty} (\mathbb D)\|<\infty.
\]
For the unit disc $D$ the embedding
operator 
\[
i_{D}:W^{1,2}(\mathbb D)\hookrightarrow L^{2}(\mathbb D)
\]
 is compact (see, for example, \cite{M}).
 
By Theorem~\ref{thm:LpLq} the composition
operator 
\[
\varphi^{\ast}:L^{2}(D)\to L^{2}(\Omega)
\]
 is bounded if 
\[ 
\|\psi' \mid L^{\infty} (\mathbb D)\|^2=\|J_{\varphi^{-1}} \mid L^{\infty} (\mathbb D)\|<\infty.
\]

Since a conformal homeomorphism $\varphi^{-1}$ induces a bounded composition
operator 
\[
(\varphi^{-1})^{\ast}:L^{1,2}(\Omega)\to L^{1,2}(D)
\]
 then by Theorem~\ref{thm:boundWW} the composition
operator 
\[
(\varphi^{-1})^{\ast}:W^{1,2}(\Omega)\to W^{1,2}(D)
\]
 is bounded. 

Therefore the embedding operator 
\[
i_{\Omega}:W^{1,2}(\Omega)\hookrightarrow L^{2}(\Omega)
\]
is compact as a composition of bounded composition operators $\varphi^{\ast}$,
$(\varphi^{-1})^{\ast}$ and the compact embedding operator $i_{D}$
\[
i_{D}:W^{1,2}(D)\hookrightarrow L^{2}(D).
\]

The first part of this theorem is proved. 

For any function  $f\in W^{1,2}(\Omega)$ and $g=f\circ\varphi^{-1}\in W^{1,2}(\mathbb D)$ the following inequalities are correct:

\begin{multline}
\biggl(\iint\limits_{\Omega}|f(x,y)-f_{\Omega}|^2~dxdy\biggr)^{\frac{1}{2}}\\=
\inf\limits_{c\in\mathbb R}\biggl(\iint\limits_{\Omega}|f(x,y)-c|^2~dxdy\biggr)^{\frac{1}{2}}
\leq \biggl(\iint\limits_{\Omega}|f(x,y)-f_{\Omega,h}|^2~dxdy\biggr)^{\frac{1}{2}}\\
=\biggl(\iint\limits_{\Omega}|f(x,y)-f_{\Omega,h}|^2J^{-1}_{\varphi}(x,y)J_{\varphi}(x,y)~dxdy\biggr)^{\frac{1}{2}}
\\
\leq \|J_{\varphi} \mid L^{\infty}(\Omega)\|^{-\frac{1}{2}}\cdot
\biggl(\iint\limits_{\Omega}|f(x,y)-f_{\Omega,h}|^2 J_{\varphi}(x,y)~dxdy\biggr)^{\frac{1}{2}}\\
= \|J_{\varphi^{-1}}\mid L^{\infty}(\mathbb D)\|^{\frac{1}{2}}\cdot
\biggl(\iint\limits_{\Omega}|f(x,y)-f_{\Omega,h}|^2 J_{\varphi}(x,y)~dxdy\biggr)^{\frac{1}{2}}\\
= \|J_{\varphi^{-1}}\mid L^{\infty}(\mathbb D)\|^{\frac{1}{2}}\cdot
\biggl(\iint\limits_{\Omega}|f(x,y)-g_{\mathbb D}|^2 J_{\varphi}(x,y)~dxdy\biggr)^{\frac{1}{2}}
\nonumber
\end{multline}
Using the change of variable formula and the Poincar\'e-Sobolev inequality in the unit disc we have
\begin{multline}
\biggl(\iint\limits_{\Omega}|f(x,y)-f_{\Omega}|^2~dxdy\biggr)^{\frac{1}{2}}\leq
\|J_{\varphi^{-1}}\mid L^{\infty}(\mathbb D)\|^{\frac{1}{2}}\cdot
\biggl(\iint\limits_{\mathbb D}|g(u,v)-g_{\mathbb D}|^2 ~dudv\biggr)^{\frac{1}{2}}\\
= \|J_{\varphi^{-1}}\mid L^{\infty}(\mathbb D)\|^{\frac{1}{2}}\cdot
\biggl(\iint\limits_{\mathbb D}|\nabla g(u,v)|^2 ~dudv\biggr)^{\frac{1}{2}}
= \|\psi'\mid L^{\infty}(\mathbb D)\|\cdot
\biggl(\iint\limits_{\Omega}|\nabla f(x,y)|^2 ~dudv\biggr)^{\frac{1}{2}}
\nonumber
\end{multline}

\end{proof}

\section{Eigenvalue Problem for Free Vibrating Membrane}

Eigenvalue problem for free vibrating membrane is equivalent to the corresponding problem for the Neumann-Laplace operator. The classical formulation for smooth domains is the following

\begin{gather}
-\Delta u=\lambda u\,\,\text{in}\,\,\Omega,\label{Neum2}\\
{\frac{\partial u}{\partial n}}\bigg|_{\partial\Omega}=0.\label{Neum12}
\end{gather}

Because conformal regular domains are not necessary smooth the weak statement of the spectral problem for the Neumann-Laplace operator is convenient: A function
$u$ solves the previous problem iff  ($u\in W^{1,2}(\Omega)$) and 
\[
\iint\limits _{\Omega}\nabla u(x,y)\cdot\nabla v(x,y)~dxdy=\lambda\iint\limits _{\Omega}u(x,y)v(x,y)~dxdy
\]
 for all  $v\in W^{1,2}(\Omega)$.

By the Min-Max Principle  \cite{D1} the inverse to the first eigenvalue is equal to the exact constant in the Poincar\'e inequality
$$
\int\limits_{\Omega}|f(x,y)-f_{\Omega}|^2~dxdy\leq B^2_{2,2}[\Omega] \int\limits_{\Omega}|\nabla f(x,y)|^2~dxdy.
$$

We are ready to prove the main result about the spectrum ({\bf Theorem A}). For readers convenience we repeat its formulation:

\vskip0.3cm

{\bf Theorem A.}
{\it Let $\Omega\subset\mathbb R^2$ be a conformal $\alpha$-regular domain. Then the spectrum of Neumann-Laplace operator in $\Omega$ is discrete, can be written in the form
of a non-decreasing sequence
\[
0=\lambda_0[\Omega]<\lambda_{1}[\Omega]\leq\lambda_{2}[\Omega]\leq...\leq\lambda_{n}[\Omega]\leq...\,,
\]
and
\begin{multline}
{1}/{\lambda_1[\Omega]}\leq B^2_{2\alpha/(\alpha-2),2}[\mathbb D]\left(\int\limits_{\mathbb D}|\varphi'(x,y)|^{\alpha}~dxdy\right)^{\frac{2}{\alpha}}\\
\leq 4\pi^{-\frac{2}{\alpha}}\left(\frac{2\alpha-2}{\alpha-2}\right)^{\frac{2\alpha-2}{\alpha}}{\|\psi'\mid L^{\alpha}(\mathbb D)\|^2}.
\nonumber
\end{multline}
where $\psi:\mathbb D\to\Omega$ is the Riemann conformal mapping of the unit disc $\mathbb D\subset\mathbb R^2$ onto $\Omega$. }

\vskip0.3cm

\begin{proof}

By Theorem \ref{thm:comembd} in the case $s=2$ the embedding operator
\[
i:W^{1,2}(\Omega)\hookrightarrow L^{2}(\Omega),\,\,\,
\]
is compact.

Therefore the spectrum of the Neumann-Laplace operator is discrete and  can be written in the form of a non-decreasing sequence. 

By the same theorem and the Min-Max principle we have

\begin{multline}
\int\limits_{\Omega}|f(x,y)-f_{\Omega}|^2~dxdy=\inf\limits_{c\in\mathbb R}\int\limits_{\Omega}|f(x,y)-f_{\Omega}|^2~dxdy
\\
\leq B^2_{2,2}[\Omega] \int\limits_{\Omega}|\nabla f(x,y)|^2~dxdy,
\nonumber
\end{multline}
where $B_{2,2}[\Omega]\leq B_{r,2}[\mathbb D]\cdot \|\psi'|L^{\alpha}(\mathbb D)\|$.

Hence

\[
{1}/{\lambda_1[\Omega]}\leq  B^2_{r,2}[\mathbb D]\left(\int\limits_{\mathbb D}|\psi'(u,v)|^{\alpha}~dudv\right)^{\frac{2}{\alpha}}.
\]

By Proposition \ref{EstCon} 
$$
B_{r,2}[\Omega,h]=B_{r,2}[\mathbb D]\leq 2\pi^{\frac{2-r}{2r}}\left({(r+2)}/{2}\right)^{\frac{r+2}{2r}}.
$$
Recall that in Theorem \ref{thm:comembd} $r=2\alpha/(\alpha-2)$. In this case
\[
B_{2\alpha/(\alpha-2),2}[\mathbb D] \leq 2\pi^{-\frac{1}{\alpha}}\left(\frac{2\alpha-2}{\alpha-2}\right)^{\frac{\alpha-1}{\alpha}}.
\]

Therefore

\begin{multline}
{1}/{\lambda_1[\Omega]}\leq B^2_{2\alpha/(\alpha-2),2}[\mathbb D]\left(\int\limits_{\mathbb D}|\varphi'(x,y)|^{\alpha}~dxdy\right)^{\frac{2}{\alpha}}\\
\leq 4\pi^{-\frac{2}{\alpha}}\left(\frac{2\alpha-2}{\alpha-2}\right)^{\frac{2\alpha-2}{\alpha}}{\|\psi'\mid L^{\alpha}(\mathbb D)\|^2}.
\nonumber
\end{multline}
\end{proof}

\vskip 0.3cm

In the case of conformal $\alpha$-regular domains for $\alpha=\infty$ by Theorem~\ref{thm:infembd} we immediately have

{\bf Theorem B.}
{\it Let $\Omega\subset\mathbb R^2$ be a conformal $\alpha$-regular domain for $\alpha=\infty$. Then the spectrum of Neumann-Laplace operator in $\Omega$ is discrete, can be written in the form
of a non-decreasing sequence
\[
0=\lambda_0[\Omega]<\lambda_{1}[\Omega]\leq\lambda_{2}[\Omega]\leq...\leq\lambda_{n}[\Omega]\leq...\,,
\]
and
\begin{equation}
\label{eq:estinf}
{1}/{\lambda_1[\Omega]}\leq  B^2_{2,2}[\mathbb D]\|\psi'\mid L^{\infty}(\mathbb D)\|^2=
\frac{\|\psi'\mid L^{\infty}(\mathbb D)\|^2}{\left(j_{1,1}^{\prime}\right)^2},
\end{equation}
where $j_{1,1}^{\prime}\approx1.84118$ is the first positive zero of the derivative of the Bessel function $J_1$ and $\psi:\mathbb D\to\Omega$ is the Riemann conformal mapping of the unit disc $\mathbb D\subset\mathbb R^2$ onto $\Omega$. 
}
\vskip 0.3cm

\subsection{Examples}

Now we describe a rather wide class of plane domains for which there exist conformal mappings with Jacobians of the class
$L^p(\mathbb D)$ for some $p>1$, i.e. with complex derivatives of the class
$L^p(\mathbb D)$ for some $p>2$.

\begin{defn}
A homeomorphism $\varphi:\Omega \rightarrow \Omega_1$
between planar domains is called $K$-quasiconformal if it preserves
orientation, belongs to the Sobolev class $W_{\loc}^{1,2}(\Omega)$
and its directional derivatives $\partial_{\alpha}$ satisfy the distortion inequality
$$
\max\limits_{\alpha}|\partial_{\alpha}\varphi|\leq K\min_{{\alpha}}|\partial_{\alpha}\varphi|\,\,\,
\text{a.e. in}\,\,\, \Omega \,.
$$
\end{defn}
Infinitesimally, quasiconformal homeomorphisms transform circles to ellipses
with eccentricity uniformly bounded by $K$. If $K=1$ we recover
conformal homeomorphisms, while for $K>1$ plane quasiconformal mappings need
not be smooth.
\begin{defn}
A domain $\Omega$ is called a $K$-quasidisc if it is the image of the
unit disc $\mathbb{D}$ under a $K$-quasiconformal homeomorphism of
the plane onto itself.
\end{defn}

It is well known that the boundary of any $K$-quasidisc $\Omega$
admits a $K^{2}$-quasi\-con\-for\-mal reflection and thus, for example,
any conformal homeomorphism $\varphi:\mathbb{D}\to\Omega$ can be
extended to a $K^{2}$-quasiconformal homeomorphism of the whole plane
to itself.

The boundaries of quasidiscs are called quasicircles. It is known that there are quasicircles for which no segment has finite length.
The Hausdorff dimension of quasicircles was first investigated by F. W. Gehring and J. V\"ais\"al\"a  \cite{GV73},
who proved that it can take all values in the interval $[1,2)$. S. Smirnov proved recently \cite{Smi10} that the Hausdorff dimension of
any $K$-quasicircle is at most $1+k^2$, where $k = (K-1)/(K +1)$.

Ahlfors's 3-point condition \cite{Ahl63} gives
a complete geometric characterization of quasicircles: a Jordan curve $\gamma$ in the plane is a quasicircle
if and only if for each two points $a, b$ in $\gamma$ the (smaller) arc between them has the
diameter comparable with $|a-b|$. This condition is easily checked for the snowflake.
On the other hand, every quasicircle can be obtained by an explicit snowflake-type
construction (see \cite{Roh01}).

For any planar $K$-quasiconformal homeomorphism $\varphi:\Omega\rightarrow \Omega_1$
the following sharp result is known: $J(z,\varphi)\in L_{\loc}^{p}(\Omega_{1})$
for any $p<\frac{K}{K-1}$ (\cite{G1, A}).
\begin{prop}
\label{prop:confQuasidisc}Any conformal homeomorphism $\varphi:\mathbb{D}\to\Omega$
of the unit disc $\mathbb{D}$ onto a $K$-quasidisc $\Omega$ belongs to $L^{1,p}(\mathbb{D})$
for any $1\le p<\frac{2K^{2}}{K^2-1}$.
\end{prop}
\begin{proof} Any conformal homeomorphism $\varphi:\mathbb{D}\to\Omega$ can be extended to a $K^2$ quasiconformal homeomorphism  $\psi$
of the whole plane to the whole plane by reflection.
Since the domain $\Omega$ is bounded, $\psi$ belongs to the class $L^p(\Omega)$ for any $1\le p<\frac{2K^2}{K^2-1}$ (\cite{G1}, \cite{A}).
Therefore $\varphi$ belongs to the same class.
\end{proof}

For quasidiscs the following estimate readily follows from Theorem A

\begin{prop} Suppose a conformal homeomorphism $\varphi:\mathbb{D}\to\Omega$ maps
the unit disc $\mathbb{D}$ onto a $K$-quasidisc $\Omega$. Then
\begin{multline}
{1}/{\lambda_1[\Omega]}\leq B^2_{2\alpha/(\alpha-2),2}[\mathbb D]\left(\int\limits_{\mathbb D}|\varphi'(x,y)|^{\alpha}~dxdy\right)^{\frac{2}{\alpha}}\\
 \leq 4\pi^{-\frac{2}{\alpha}}\left(\frac{2\alpha-2}{\alpha-2}\right)^{\frac{2\alpha-2}{\alpha}}{\|\psi'\mid L^{\alpha}(\mathbb D)\|^2}
\nonumber
\end{multline}
for any $2<\alpha<\frac{2K^2}{K^2-1}.$

\end{prop}

As the second example we consider the interior of the cardioid. By the Alhfors condition the cardioid is not a quasidisc.
Because the cardioid is a conformal $\infty$-regular domain, we have the following example:

\begin{exa}
\label{exa:cardioid}
Let $\Omega_c$ be the interior of the cardioid. The diffeomorphism
$$
z=\psi(w)=(w+1)^2,z=x+iy,
$$
is conformal and maps the unit disc $\mathbb D$ onto $\Omega_c$. 
Then by Theorem B
$$
\|\psi'\mid L^{\infty}(\mathbb D)\|=\max\limits_{w\in\mathbb D}2|w+1|\leq 4.
$$
Hence
$$
\lambda_1[\Omega_c]\geq \frac{\left(j_{1,1}^{\prime}\right)^2}{16}.
$$
Here $j_{1,1}^{\prime}\approx1.84118$ is the first positive zero of the derivative of the Bessel function $J_1$.
\end{exa}

The third example is $m$-polygon $P_{m}$.

Consider a $m$-polygon $P_{m}$ with vertices $z_{k}$ on the unit
circle and the angles $\alpha_{k}$ are measured in fractions of $\pi$.
Then the conformal mapping of the unit disc $\mathbb{D}$ onto the
$n$-polygon $P_{n}$ is given by the Schwarz-Christoffel formula:
\[
\psi(z)=C\int\limits _{z_{0}}^{z}(z-z_{1})^{\alpha_{1}-1}(z-z_{2})^{\alpha_{2}-1}\cdot...\cdot(z-z_{m})^{\alpha_{m}-1}dz+C_{1}
\]
 and 
\[
\psi'(z)=C(z-z_{1})^{\alpha_{1}-1}(z-z_{2})^{\alpha_{2}-1}\cdot...\cdot(z-z_{m})^{\alpha_{m}-1}.
\]

\begin{exa}
\label{exa:square} $Q_m$ is a regular $m$-polygon with vertices $z_k$ on the unit circle and the angles $\alpha_k=1-2/m$ are  measured in fractions of $\pi$.  The diffeomorphism 
\[
\psi(z)=C\int\limits _{z_{0}}^{z}(z^{m}-1)^{-\frac{2}{m}}~dz+C_{1},\,\,\, z=x+iy,
\]
 is conformal and maps the unit disc $\mathbb{D}$ onto the regular $m$-polygon
$Q_m$. If we have $\varphi(0)=0$ and $|\varphi'(z)|=1$, then $C_{1}=0$
and $C=1$.
Then 
$$
\frac{1}{\lambda_1[\Omega]}\leq  \pi^{-\frac{2}{\alpha}}\inf\limits_{2<\alpha< m}\left(\frac{2\alpha-2}{\alpha-2}\right)^{\frac{2\alpha-2}{\alpha}}
\iint\limits_{\mathbb D}|z^{m}-1|^{-\frac{2\alpha}{m}}~dxdy.
$$
\end{exa}

\section{Estimates for domains conformallly equivalent to a rectangle}

We take the unit ball as the basic domains for our estimates. In many applications more convenient to take as the basic domain a rectangle:
$$
\mathbb Q_{ab}=\{(x,y)\in \mathbb R^2: 0<x<a, 0<y<b\}
$$

In this case we have the following assertion:

\vskip 0.3cm
{\bf Theorem C.}
{\it Let $\Omega\subset\mathbb R^2$ be a plane domain with non-empty boundary. Suppose that there exists a conformal mapping $\psi: \mathbb Q_{ab}\to\Omega$ such that $\psi'\in L^{\alpha}(\mathbb Q_{ab})$ for some $\alpha>2$. Then the spectrum of Neumann-Laplace operator in $\Omega$ is discrete, can be written in the form
of a non-decreasing sequence
\[
0=\lambda_0[\Omega]<\lambda_{1}[\Omega]\leq\lambda_{2}[\Omega]\leq...\leq\lambda_{n}[\Omega]\leq...\,,
\]
\begin{multline}
{1}/{\lambda_1[\Omega]}\leq B^2_{2\alpha/(\alpha-2),2}[\mathbb Q_{ab}]\left(\int\limits_{\mathbb Q_{ab}}|\varphi'(x,y)|^{\alpha}~dxdy\right)^{\frac{2}{\alpha}}\\
\leq \left(\frac{a^2+b^2}{(ab)^{\frac{r-1}{r}}}\right)^2\left(\frac{2\alpha-2}{\alpha-2}\right)^{\frac{2\alpha-2}{\alpha}}{\|\psi'\mid L^{\alpha}(\mathbb Q_{ab})\|^2} \,\,\text{for}\,\,\alpha<\infty
\nonumber
\end{multline}
and 
\begin{multline}
{1}/{\lambda_1[\Omega]}\leq  B^2_{2,2}[\mathbb Q_{ab}]\|\psi'\mid L^{\infty}(\mathbb Q_{ab})\|^2\\=
\left(\frac{\max\{a,b\}}{\pi}\right)^2\|\psi'\mid L^{\infty}(\mathbb Q_{ab})\|^2,\,\,\text{for}\,\,\alpha=\infty.
\nonumber
\end{multline}
}

\begin{proof}
 Because $\mathbb Q_{ab}$ is the convex domain, then by Proposition~\ref{EstCon} we have that
\begin{multline}
B_{r,2}(\mathbb Q_{ab})\leq \frac{d^2}{2|\mathbb Q_{ab}|}\left(\frac{1-1/2+1/r}{1/2-1/2+1/r}\right)^{1-1/2+1/r}\omega_n^{1-1/2}|\mathbb Q_{ab}|^{1/2-1/2+1/2}\\
=\left(\frac{a^2+b^2}{(ab)^{\frac{r-1}{r}}}\right)^2\left(\frac{2\alpha-2}{\alpha-2}\right)^{\frac{2\alpha-2}{\alpha}}.
\nonumber
\end{multline}

For $B_{2,2}(\mathbb Q_{ab})$ there is the exact calculation (see, for example, \cite{{NKP}})
$$
B_{2,2}(\mathbb Q_{ab})=\frac{\max\{a,b\}}{\pi}.
$$

Now, replacing in the proof of Theorem~\ref{thm:comembd} and Theorem~\ref{thm:infembd}  the unit disc $\mathbb D$ by the 
rectangle $\mathbb Q_{ab}$ and using the well known fact that $\mathbb Q_{ab}$ is Poincar\'e domain we prove Theorem C similarly to proof Theorem A and Theorem~B.
\end{proof}
\vskip0.3cm

\begin{exa}
\label{exa:expo} Let $a=1$ and $b=2\pi$. The conformal mapping $\psi=e^w:\mathbb Q_{ab}\to R_e$ maps the rectangle $\mathbb Q_{ab}$  onto the split ring $R_e$. 
Then, by Theorem C 
$$
\|\psi'\mid L^{\infty}(\mathbb D)\|=\max\limits_{u\in (0,1)}e^u\leq e.
$$
Hence
$$
\lambda_1[R_e]\geq \left(\frac{\pi}{2\pi}\right)^2\frac{1}{\|\psi'\mid L^{\infty}(\mathbb D)\|^2}=\frac{1}{4e^2}.
$$
\end{exa}

Finally we note, that using the elementary conformal functions like $\sin$, $\tan$ and so on, it is possible to construct many non-trivial examples of domains in which Theorem B gives estimates of the first non-trivial eigenvalue of the Neumann-Laplace operator.

\section{Comparison of the estimates with previous results}

Let $\mathbb D_a$ denote the disc of the radius $a>0$. Then for $\mathbb D_a$  our estimates is exact. In this case 
$$
\mathbb D_a=\psi(\mathbb D),
$$
where $\psi(z)=az$. By Theorem B:
$$
\lambda_1[\mathbb D_r]\geq \frac{\left(j_{1,1}^{\prime}\right)^2}{\|\psi'\mid L^{\infty}(\mathbb D)\|^2}=\frac{\left(j_{1,1}^{\prime}\right)^2}{a^2}.
$$ 

In the paper \cite{PW} the authors proved that if $\Omega$ is convex with diameter $d(\Omega)$ (see, also \cite{ENT, FNT}), then
\begin{equation}
\label{eq:PW}
\lambda_1[\Omega]\geq \frac{\pi^2}{d(\Omega)^2}.
\end{equation}

\begin{defn}
Let $\Omega=\psi(\mathbb D)$ be a conformal $\alpha$-regular domain for $\alpha=\infty$. 
We call a domain $\Omega$ as {\it a conformal  uniform domain} if
$$
\|\psi'\mid L^{\infty}(\mathbb D)\|< \frac{j_{1,1}^{\prime}}{\pi}d(\Omega).
$$
\end{defn}

For the class of conformal uniform domains the estimate (\ref{eq:estinf}) (Theorem B) improves the estimate (\ref{eq:PW}):
$$
\lambda_1[\Omega]\geq \frac{\left(j_{1,1}^{\prime}\right)^2}{\|\psi'\mid L^{\infty}(\mathbb D)\|^2}> \frac{\pi^2}{d(\Omega)^2}.
$$

The class of conformal uniform domains is not empty. Consider domains 
$$
\Omega_n=\psi_n(\mathbb D),\,\, \psi_n(z)=(z+n)^2, \,\,n\geq 1.
$$
Then 
$$
\|\psi'_n\mid L^{\infty}(\mathbb D)\|=\max\limits_{|z|\leq 1}|2(z+n)|=2(n+1).
$$
From another side 
$$
d(\Omega)\geq |\psi_n(1)-\psi_n(-1)|=|(n+1)^2-(n-1)^2|=4n.
$$

Hence $\|\psi'_n\mid L^{\infty}(\mathbb D)\|< \frac{j_{1,1}^{\prime}}{\pi}d(\Omega_n)$ for $n>5$ and $\Omega_n$ are conformal uniform domains. For $n>3$ the domains $\Omega_n$ are convex, because 
$$
Re\left\{1+z\frac{\psi''_n}{\psi'_n}\right\}>0,\,\,n>3.
$$

The conformal uniform domains can be characterized in geometric terms: a conformal $\infty$-regular domain $\Omega$ is conformal uniform iff
$$
R\left(\psi(z),\Omega\right)\leq \left(1-|z|^2\right)d(\Omega),
$$
where $R\left(\psi(z),\Omega\right)$ is a conformal radius of $\Omega$.

Another example of conformal uniform domains is the domains $\Omega_{\alpha}=\psi_{\alpha}(\mathbb D)$, 
where $\psi(z)=e^{\alpha z}$, for $0=\alpha_0<\alpha<\alpha_1$, where $\alpha_0$ and $\alpha_1$ are the zeros of the function $f(\alpha)=1-\frac{\pi}{j_{1,1}^{\prime}}\alpha-e^{-2\alpha}$. For these mappings $\psi_{\alpha}$:
$$
\|\psi'_{\alpha}\mid L^{\infty}(\mathbb D)\|=\max\limits_{|z|\leq 1}|\alpha e^{\alpha z}|=\alpha e^{\alpha},
$$
and 
$$
d(\Omega)\geq |\psi_{\alpha}(1)-\psi_{\alpha}(-1)|=e^{\alpha}-e^{-\alpha}.
$$
Hence the domains $\Omega_{\alpha}$ are conformal uniform domains if 
$$
\alpha e^{\alpha}\leq \frac{j_{1,1}^{\prime}}{\pi}\left(e^{\alpha}-e^{-\alpha}\right),
$$
or 
$$
f(\alpha)=1-\frac{\pi}{j_{1,1}^{\prime}}\alpha-e^{-2\alpha}\geq 0.
$$

For these $\alpha$ the domains $\Omega_{\alpha}$ are convex, because
$$
Re\left\{1+z\frac{\psi''_{\alpha}}{\psi'_{\alpha}}\right\}=Re\left\{1+\alpha z\right\}>0,\,\, for\,\, 0=\alpha_0<\alpha<\alpha_1.
$$

\noindent Vladimir Gol'dshtein \, \hskip 3.2cm Alexander Ukhlov

\noindent Department of Mathematics \hskip 2.25cm Department of Mathematics

\noindent Ben-Gurion University of the Negev \hskip 1.05cm Ben-Gurion
University of the Negev

\noindent P.O.Box 653, Beer Sheva, 84105, Israel \hskip 0.7cm P.O.Box
653, Beer Sheva, 84105, Israel

\noindent E-mail: vladimir@bgu.ac.il \hskip 2.5cm E-mail: ukhlov@math.bgu.ac.il
\end{document}